\documentclass[10pt]{amsart}
\usepackage{amssymb}
\usepackage{epsfig}
\usepackage{url}
\usepackage{setspace}
\usepackage{color}
\theoremstyle{plain}

\newtheorem{thm}{Theorem}[section]
\newtheorem{cor}[thm]{Corollary}

\newtheorem{rem}[thm]{Remark}
\newtheorem{ques}[thm]{Question}
\newtheorem{conj}[thm]{Conjecture}

\def\cal{\mathcal}
\def\bbb{\mathbb}
\def\op{\operatorname}
\renewcommand{\phi}{\varphi}
\newcommand{\R}{\bbb{R}}

\begin{document}
\title[Rational solutions of equations involving norms]{Rational solutions of certain Diophantine equations involving norms}
\author{Maciej Ulas}

\thanks{Research of the author was supported by Polish Government funds for science, grant IP 2011 057671 for the years 2012--2013.}

\begin{abstract}
In this note we present some results concerning the unirationality of the algebraic variety $\cal{S}_{f}$ given by the equation
\begin{equation*}
N_{K/k}(X_{1}+\alpha X_{2}+\alpha^2 X_{3})=f(t),
\end{equation*}
where $k$ is a number field, $K=k(\alpha)$, $\alpha$ is a root of an irreducible polynomial $h(x)=x^3+ax+b\in k[x]$ and $f\in k[t]$. We are mainly interested in the case of pure cubic extensions, i.e. $a=0$ and $b\in k\setminus k^{3}$. We prove that if $\op{deg}f=4$ and the variety $\cal{S}_{f}$ contains a $k$-rational point $(x_{0},y_{0},z_{0},t_{0})$ with $f(t_{0})\neq 0$, then $\cal{S}_{f}$ is $k$-unirational. A similar result is proved for a broad family of quintic polynomials $f$ satisfying some mild conditions (for example this family contains all irreducible polynomials). Moreover, the unirationality of $\cal{S}_{f}$ (with non-trivial $k$-rational point) is proved for any polynomial $f$ of degree 6 with $f$ not equivalent to the polynomial $h$ satisfying the condition $h(t)\neq h(\zeta_{3}t)$, where $\zeta_{3}$ is the primitive third root of unity. We are able to prove the same result for an extension of degree 3 generated by the root of polynomial $h(x)=x^3+ax+b\in k[x]$,  provided that $f(t)=t^6+a_{4}t^4+a_{1}t+a_{0}\in k[t]$ with $a_{1}a_{4}\neq 0$.

\end{abstract}

\maketitle

\section{Introduction}\label{sec1}
Let $k$ be a number field and $K/k$ be an algebraic extension of degree $n$. There is a lot of papers devoted to the study of $k$-rational solutions of Diophantine equations of the form
\begin{equation}\label{hyp}
N_{K/k}(X_{1}\omega_{1}+\ldots+X_{n}\omega_{n})=f(t),
\end{equation}
where $N_{K/k}$ is a full norm form for the extension $K/k$, $\{\omega_{1},\ldots,\omega_{n}\}$ is a fixed basis of the extension and $f(t)$ is a polynomial over $k$. The main problem here is the question whether the Hasse principle, or in other words local to global principle, holds for the smooth proper model of a hypersurface given by the equation (\ref{hyp}). For example, if $f(t)$ is constant then the local to global principle holds for (\ref{hyp}) (Hasse). If $n=2$ and $\op{deg}f=3$ or $4$ then the variety defined by (\ref{hyp}) is called a Ch\^{a}telet surface. The arithmetic of these surfaces is well understood. In particular, in \cite{CTSSD1, CTSSD2} it is proved that the Brauer-Manin obstruction for the Hasse principle and weak approximation is the only one. Moreover, the existence of a $k$-rational solution implies $k$-unirationality. These results are unconditional. However, the most general result in this area is obtained under Schinzel's hypothesis (H) and says that if $K$ is a cyclic extension of a number field $k$, and $f(t)$ is a separable polynomial of arbitrary degree, then the Brauer-Manin obstruction to the Hasse principle and weak approximation is the only one for the smooth and projective model $X$ of the variety given by the equation (\ref{hyp}). Moreover, if there is no Brauer-Manin obstruction to the Hasse principle then the $k$-rational points are Zariski dense in $X$.

Most of the results in this area were proved using algebraic considerations (via the computation of the Brauer-Manin obstructions) or a combination of algebraic methods together with analytic techniques (see for example \cite{HS}). However, only few papers present constructions which allow to produce new solutions from a given $k$-rational solution of (\ref{hyp}). As was mentioned in \cite[p.162]{HS}, usually this is a rather difficult problem.

We are working with a field $k$ of characteristic 0 and an algebraic extension $K/k$ of degree $n$. We take $\omega_{i}=\alpha^{i-1}$ for $i=1,\ldots,n$, where $\alpha\in K$ is chosen in such a way that $K=k(\alpha)$. We thus are interested in the equation
\begin{equation}\label{generalequation}
\op{N}_{K/k}(X_{1},\ldots,X_{n})=f(t),
\end{equation}
where in order to shorten the notation we put
\begin{equation*}
\op{N}_{K/k}(X_{1},\ldots,X_{n}):=N_{K/k}(X_{1}+\alpha X_{2}\ldots+\alpha^{n-1} X_{n}),
\end{equation*}
i.e. $\op{N}_{K/k}$ will denote a norm form, and $N_{K/k}$ denotes the corresponding field norm. In the sequel by a {\it non-trivial} solution of (\ref{generalequation}) we mean a solution $(X_{1},\ldots,X_{n},t)$ which satisfies $f(t)\neq 0$. The present paper is a contribution to the subject in that we show that in some cases the existence of one $k$-rational solution of (\ref{generalequation}) implies the existence of infinitely many $k$-rational solutions. This is obtained mainly by constructing parametric solution of the corresponding equation, or, in a more geometric language, by constructing of a $k$-rational curve lying on the corresponding algebraic variety. Of course, we are interested only in the existence of $k$-rational curves which are not contained in the fiber of the map $\Phi: \cal{S}_{f}\ni (X_{1},\ldots,X_{n},t)\mapsto t\in \mathbb{P}^{1}(k)$. Our argument is based on a similar approach proposed by Mestre in a series of papers \cite{Mes1, Mes2, Mes3} devoted to the study of the existence of rational points on (generalized) Ch\^{a}telet surfaces, i.e. surfaces defined by (\ref{generalequation}) with $n=2$ and $\op{deg}f\geq 5$.

Let us describe the content of the paper in some details. In Section \ref{sec2}, we prove that if $K/k$ is a pure cubic extension generated by the root of the polynomial $h(x)=x^3+b\in k[x]$, $f\in k[t]$ is of degree 4, and the variety $\cal{S}_{f}$ defined by the equation (\ref{generalequation}) contains a non-trivial $k$-rational point, then $\cal{S}_{f}$ is unirational over $k$. In particular, in this case the set of $k$-rational points on $\cal{S}_{f}$ is Zariski dense. We prove a similar result for $f\in k[t]$ of degree 5, provided that $f$  satisfies some mild conditions. In particular, if $f$ is an irreducible polynomial, then $\cal{S}_{f}$ is $k$-unirational. We also prove that if $f\in k[t]$ is a monic polynomial of degree 6, $S_{f}$ contains non-trivial $k$ rational point and the polynomial $f$ is not equivalent to a polynomial $h\in k[t]$ satisfying the condition $h(t)\neq h(\zeta_{3}t)$, then the variety $\cal{S}_{f}$ given by equation (\ref{generalequation}) is $k$-unirational. This result is particulary interesting in the light of a recent work of V\'{a}rilly-Alvarado and Viray \cite{VarBir}. Indeed, in the case under consideration the variety $\cal{S}_{f}$ is a so called {\it Ch\^{a}telet threefold} (in the terminology of \cite{VarBir}). The authors of the cited paper asked whether the existence of a $k$-rational point on $\cal{S}_{f}$ implies $k$-unirationality \cite[Problem 6.2]{VarBir}. Our result shows that $\cal{S}_{f}$ is $k$-unirational for a broad class of polynomials. Moreover, if $k$ is a number field with a real embedding, we prove that for each polynomial $f(t)=a_{0}t^6+\sum_{i=0}^{4}a_{6-i}t^i\in k[t]$ and any given $\epsilon >0$ there exists a polynomial $g(t)=c_{0}t^6+\sum_{i=0}^{4}c_{6-i}t^i\in k[t]$ which is close to $f$, i.e. $|a_{i}-c_{i}|<\epsilon$ for $i=0,2,\ldots,6$, and such that for any $b\in k\setminus k^{3}$ and a pure cubic extension $K/k$ generated by the root of the polynomial $h(x)=x^3+b$, the variety $\cal{S}_{g}$ is unirational over $k$.

In Section \ref{sec3}, we consider the variety $\cal{S}_{f}$ defined by the equation (\ref{generalequation}) involving a norm form of an extension $K/k$ generated by the root of an irreducible polynomial $h(x)=x^3+ax+b\in k[x]$. We prove that if $f(t)=t^{6}+a_{4}t^4+a_{1}t+a_{0}\in k[t]$, $a_{1}a_{4}\neq 0$ then the variety $\cal{S}_{f}$ is unirational over $k$. Moreover, we give a remark concerning the unirationality of slightly more general varieties defined by equations of the form $F(x,y,z)=f(t)$, where $F$ is a homogenous form of degree 3 and $f$ is a polynomial.

\section{Solutions of the equation $\op{N}_{K/k}(X_{1},X_{2},X_{3})=f(t)$ with pure cubic extension $K/k$ and $f$ of degree $\leq 6$}\label{sec2}

Let $k$ be a field of characteristic 0 and $K/k$ be an extension of degree 3 generated by the root, say $\alpha$, of the irreducible polynomial $h(x)=x^3+ax+b$ defined over $k$. We are interested in the rational points lying on the variety defined by the equation
\begin{equation}\label{degree3}
\cal{S}_{f}:\;\op{N}_{K/k}(X_{1},X_{2},X_{3})=f(t),
\end{equation}
where  $f\in k[t]$. In this section we consider the case of $f$ of degree $\leq 6$. Since we are interested in $k$-unirationality of $\cal{S}_{f}$, we make the assumption that the set of $k$-rational points on $\cal{S}_{f}$ is non empty. To be more precise, we assume that there is a nontrivial $k$-rational point lying on $\cal{S}_{f}$, i.e. there is a $P=(x_{0},y_{0},z_{0},t_{0})\in \cal{S}_{f}(k)$ such that $f(t_{0})\neq 0$. In particular the point $P$ is a smooth point on $\cal{S}_{f}$. In this section we consider the case of a pure cubic extension $K/k$, i.e. $K$ is generated by the root of a polynomial $h$ with $a=0$. Let us recall that in this case
\begin{equation*}
\op{N}_{K/k}(X_{1},X_{2},X_{3})=X_{1}^3-bX_{2}^3+b^2X_{3}^3+3bX_{1}X_{2}X_{3}.
\end{equation*}
Before we state our results let us note that $\cal{S}_{f}$ is isomorphic with $\cal{S}_{g}$ , where $g(t)=\sum_{i=1}^{6}c_{i}t^i+1$. Indeed, making a change of variables  $t\mapsto t+t_{0}$ we can assume that $f(0)=c_{0}=\op{N}_{K/k}(u,v,w)\neq 0$ for some $u,v,w\in k$. Multiplying this equation by $c_{0}^{-1}=\op{N}_{K/k}(u',v',w')$, with $u',v',w'$ chosen in such a way that $\op{N}_{K/k}(u,v,w)\op{N}_{K/k}(u',v',w')=1$, and using the multiplicative property of a norm form, we get the desired form of our equation. It is clear that $S_{f}$ is $k$-unirational if and only if $S_{g}$ is $k$-unirational.

We are ready to prove the following result.

\begin{thm}\label{degree4}
Let $k$ be a field of characteristic 0 and let $K=k(\alpha)$, where $\alpha^3+b=0$ with $b\in k\setminus k^{3}$. Put $g(t)=1+\sum_{i=1}^{6}c_{i}t^i\in k[t]$ and let us suppose that
\begin{equation}\label{except}
(c_{2},c_{4},c_{5})\neq \left(\frac{5 c_1^2}{12},\;-\frac{1}{144} c_1(5 c_1^3-72 c_3),\;-\frac{1}{144} c_1^2(c_1^3-12c_3)\right).
\end{equation}
Then the variety $\cal{S}_{g}$ is $k$-unirational.
\end{thm}
\begin{proof}
Let $G=G(X_{1},X_{2},X_{3},t)$ be a polynomial defining the variety $\cal{S}_{g}$. We note that $\cal{S}_{g}$ contains the $k$-rational point $(1,0,0,0)$. We use it in order to construct a $k$-rational curve lying on $\cal{S}_{g}$. More precisely, we are looking for a rational curve, say $\cal{L}$, lying on $\cal{S}_{g}$. We assume thet $\cal{L}$ can be parameterized by rational functions with parameter $u$ in the following way:
\begin{equation}\label{sub1:deg5}
\cal{L}:\;X_{1}=pT^2+qT+1,\quad X_{2}=rT^2,\quad X_{3}=sT^2+uT,\quad t=T,
\end{equation}
where $p, q, r, s, T$ need to be determined. With $X_{i}$ and $t$ defined above, we get $G(X_{1},X_{2},X_{3},t)=\sum_{i=1}^{6}C_{i}T^i$, where
\begin{equation*}
\begin{array}{lll}
  C_{1}=3q-c_{1},                  &  & C_{3}=b^2u^3+3bru+6pq+q^3-c_{3}, \\
  C_{2}=3p+3q^2-c_{2},             &  & C_{4}=3(b^2su^2+bqru+brs+p^2+pq^2)-c_{4}
\end{array}
\end{equation*}
and $C_{5}, C_{6}\in k[p,q,r,s,u]$ depend on $c_{i}$ for $i=1,\ldots,5$. The system $C_{1}=C_{2}=C_{3}=C_{4}=0$ has exactly one solution with respect to $p,q,r,s$ and it is given by:
\begin{equation*}
\begin{array}{lll}
  p=\frac{1}{9} \left(3 c_2-c_1^2\right),                  &  & r=\frac{-27 b^2 u^3+5 c_1^3-18 c_1 c_2+27 c_3}{81 b u}, \\
                                                           &  &\\
  q=\frac{1}{3}c_{1},             &  & s=\frac{u \left(27 b^2 c_1 u^3-5 c_1^4+27 c_2 c_1^2-27 c_3 c_1-27 c_2^2+81 c_4\right)}{3 \left(54 b^2 u^3+5 c_1^3-18
   c_1 c_2+27 c_3\right)}.
\end{array}
\end{equation*}
For $p, q, r, s$ defined in this way we get $C_{i}=A_{i}/D, i=5,6,$ and $DG(X_{1},X_{2},X_{3},T)=A_{5}T^5+A_{6}T^6$ for $A_{5}, A_{6}\in k[u]$ and $D=3^{12}b^2u^3(54b^2u^3+5c_1^3-18c_1c_2+27c_3)^{3}$. We note that $\op{deg}_{u}A_{6}=18$ and the leading coefficient of $A_{6}$ is $2^{3}3^{18}b^{12}$. In particular $A_{6}\neq 0$ as an element of $k[u]$. We also have $\op{deg}_{u}A_{5}=15$ and $A_{5}\neq 0$ as an element of $k[u]$ if and only if the condition (\ref{except}) is satisfied. In this case, we get a single non-zero solution of the equation $T^5(A_{5}+A_{6}T)=0$ with respect to $T$. Indeed, we have
\begin{equation*}
T=-\frac{A_{5}}{A_{6}}=\phi(u)=\frac{2\cdot 3^{19}b^{10}(5 c_1^2-12 c_2)u^{15}+\mbox{lower order terms in}\;u}{2^{3}3^{18}b^{12}u^{18}+\mbox{lower order terms in}\;u}.
\end{equation*}

Summing up, we see that the existence of a $k$-rational point $P$ with $f(t_{0})\neq 0$ implies that $\cal{S}_{g}$ contains a $k$-rational curve $\cal{L}$, which is not contained in any hyperplane defined by the equation $t=t_{0}$ with $t_{0}\in k$. This allows us to define the base change $t=\phi(u)$ which gives the cubic surface $\cal{S}_{g\circ \phi}$ defined over the field $k(u)$ with a smooth $k(u)$-rational point. This immediately implies $k(u)$-unirationality of $\cal{S}_{g\circ\phi}$ by \cite[Proposition 1.3]{CSal} and thus $k$-unirationality of $\cal{S}_{g}$. Indeed, the map $\Psi$ which guarantees unirationality of $\cal{S}_{g\circ\phi}$ extends to a dominant rational map $(\Psi,\phi)$ which gives unirationality of $\cal{S}_{g}$ and thus $S_{f}$.
\end{proof}


\begin{cor}\label{cor:degree4}
Let $k$ be a field of characteristic zero and let $K/k$ be a pure cubic extension. Consider the variety $\cal{S}_{f}$ with $f\in k[t]$ of degree 4 and suppose that  $\cal{S}_{f}$ contains a nontrivial $k$-rational point. Then $S_{f}$ is $k$-unirational.
\end{cor}
\begin{proof}
We are working with $S_{g}$ where $g(t)=1+\sum_{i=1}^{4}c_{i}t^{i}$ with $c_{4}\neq 0$. We have $S_{g}\simeq S_{f}$. In order to get the result we need to check whether the condition (\ref{except}) is satisfied for all $c_{i}\in k$ for $i=1,2,3,4$. We see that (\ref{except}) is not satisfied if and only if $(c_{2},c_{4},c_{5})=(5c_{1}^2/12,c_{1}^4/144,0)$. In particular $c_{1}\neq 0$. Making the (invertible) substitution $t\mapsto 6t/c_{1}$ we are left with the problem of proving unirationality of $S_{h}$ with $h(t)=(3t^2+2t+1)^2$. We assume that $\cal{L}$ can be parametrized by rational functions with parameter $u$ in the following way
\begin{equation}\label{sub1:deg4}
\cal{L}:\;X_{1}=T+1,\quad X_{2}=uT,\quad X_{3}=pT,\quad t=qT,
\end{equation}
where parameters $p, q, T$ still need to be determined.  For $X_{1}, X_{2}, X_{3}, t$ defined in this way we get $F=\sum_{i=1}^{4}C_{i}T^i$, where
\begin{equation*}
C_{1}=3-6q,\; C_{2}=3+3bpu-15q^2,\; C_{3}=1 + b^2 p^3 + 3 b p u - b u^3 - 18q^3,\; C_{4}=-9q^4.
\end{equation*}
We solve the system $C_{1}=C_{2}=0$ with respect to $p, q$ and get $p=1/4bu, q=1/2$. This substitution allows us to find the expression for $T$ in the form
\begin{equation*}
T=\frac{-64 b^2 u^6-32 b u^3+1}{36 b u^3}.
\end{equation*}
The expression for $T$ together with the expressions for $p, q$ give equations (\ref{sub1:deg4}) defining the rational  parametric curve $\cal{L}$ lying on $\cal{S}_{h}$. Using now the same reasoning as at the end of the proof of Theorem \ref{degree4}, we get the result.
\end{proof}

\begin{rem}{\rm We were trying to prove $k$-unirationality of $\cal{S}_{g}$ in the case when the polynomial $g\in k[t]$ is of degree 5 and does not satisfy the condition (\ref{except}). Among other things we were trying to replace the polynomial $g(t)$ by the polynomial $h(Y)=(1+vY)^6g(Y/(1+vY))$. In this way we got the variety $S_{h}$ via the substitution $X_{i}=Y_{i}/(1+vY)^2$ for $i=1,2,3$ and $t=Y/(1+vY)$. Unfortunately, one can check that if $g$ does not satisfy the condition (\ref{except}), then $h(T)$ does not satisfy the condition (\ref{except}), too. Because all our efforts failed, we decided to state the following:

\begin{ques}
Let $k$ be a field of characteristic 0 and let $K=k(\alpha)$, where $\alpha^3+b=0$ with $b\in k\setminus k^{3}$. Put $g(t)=1+\sum_{i=1}^{5}c_{i}t^i\in k[t]$ with $c_{5}\neq 0$ and let us suppose that the condition {\rm (\ref{except})} is not satisfied. Is the variety $\cal{S}_{g}$ unirational over $k$?
\end{ques}

Let us note that if the polynomial $g$ does not satisfy the condition (\ref{except}), then $g$ is reducible, namely
\begin{equation*}
g(t)=-\frac{1}{144}(c_{1}^2t^2+6c_{1}t+12)((c_{1}^3-12c_{3})t^3-c_{1}^2t^2-6c_{1} t-12).
\end{equation*}
In particular, Theorem \ref{degree4} implies that if $g$ is irreducible of degree 5 then $\cal{S}_{g}$ is $k$-unirational and thus the set of $k$-rational points on $\cal{S}_{g}$ is Zariski dense. It is clear that the same is true for a polynomial $f$ corresponding to $g$.
}
\end{rem}

In a recent paper V\'{a}rilly-Alvarado and Viray \cite{VarBir} introduced the notion of a Ch\^{a}telet threefold, which is a variety defined by the equation (\ref{generalequation}) with $n=3$ and $f\in k[t]$ of degree 6. The authors of this paper asked, whether the existence of a $k$-rational point on $\cal{S}_{f}$ implies $k$-unirationality of $\cal{S}_{f}$ \cite[Problem 6.2]{VarBir}. The statement of the Theorem \ref{degree4} gives us a broad family of polynomials $f$ such the variety $\cal{S}_{f}$ is $k$-unirational. In the next corollary we make this result more explicit.

Before we state our result, let us recall that two polynomials $f_{1}, f_{2}\in k[t]$ are equivalent if $\op{deg}f_{1}=\op{deg}f_{2}$ and there exist $\alpha, \beta\in k$ such that $f_{2}(t)=f_{1}(\alpha t+\beta)$.

\begin{cor}\label{cor:degree6}
Let $k$ be a field of characteristic 0 and let $K=k(\alpha)$, where $\alpha^3+b=0$ with $b\in k\setminus k^{3}$. Let $f\in k[t]$ be of degree 6 and suppose that $f$ is not equivalent to the polynomial $h\in k[t]$ satisfying the condition $h(t)=h(\zeta_{3}t)$, where $\zeta_{3}$ is the primitive third root of unity. Let us also suppose that $\cal{S}_{f}$ contains a nontrivial $k$-rational point. Then the variety $\cal{S}_{f}$ given by {\rm (\ref{degree3})} is $k$-unirational.
\end{cor}
\begin{proof}
First of all let us note that the existence of a non-trivial $k$-rational point on $\cal{S}_{f}$ with $f$ of degree 6 and the fact that the norm form is multiplicative, implies that $\cal{S}_{f}\simeq \cal{S}_{h}$, where $h(t)=t^{6}+\sum_{i=0}^{4}c_{6-i}t^i$ for some $c_{j}\in k, j=2,3,\ldots,6$. From our assumption on $f$ we know that at least one among the elements $c_{2},c_{4}, c_{5}$ is non-zero. Making the change of variables $X_{i}=Y_{i}/T$ for $i=1,2,3$ and $t=1/T$ we get that $S_{f}\simeq S_{g}$ with $g(T)=1+\sum_{i=2}^{6}c_{i}T^{i}$. We can apply now Theorem \ref{degree4} to the variety $S_{g}$. It is $k$-unirational provided that the condition (\ref{except}) is satisfied. In our case we have $c_{1}=0$ and thus (\ref{except}) is not satisfied if and only if $c_{2}=c_{4}=c_{5}=0$ which is not the case.
\end{proof}

Using the corollary above in the case of a number field $k$ with real embedding in $\R$, we deduce the following interesting result.

\begin{thm}\label{approx}
Let $k$ be a number field with $k\subset \R$ and put $f(t)=a_{6}t^6+\sum_{i=0}^{4}a_{6-i}t^i\in k[t]$ with $a_{0}\neq 0$. Then, for each $\epsilon>0$ there exists a polynomial $g(t)=c_{0}t^6+\sum_{i=0}^{4}c_{6-i}t^i\in k[x]$ such that $|a_{i}-c_{i}|<\epsilon$ for $i=0,2,\ldots, 6$ and for each pure cubic extension $K/k$ of degree 3, the variety $\cal{S}_{g}$ given by the equation $\op{N}_{K/k}(X_{1},X_{2},X_{3})=g(t)$ is $k$-unirational.
\end{thm}
\begin{proof}
We are working with $S_{h}\simeq S_{f}$, where $h(t)=t^6f(1/t)$. We note that for any given $a_{0}\in k^{*}$ we can find a triple $u,v,w\in k$ such that $|\op{N}_{K/k}(u,v,w)-a_{6}|<\epsilon$ and $\op{N}_{K/k}(u,v,w)\neq 0$, which is a consequence of the density of the image of the norm map $\op{N}_{K/k}:k^3\rightarrow k$. Then we take $c_{0}=\op{N}_{K/k}(u,v,w)$. If $h(t)\neq h(\zeta_{3}t)$ we take $c_{i}=a_{i}$ for $i=2,\ldots,6$. If $h(t)=h(\zeta_{3}t)$ then we take $c_{i}=a_{i}$ for $i=3,6$ and we take $c_{2}=c_{4}=c$ for any $c\in k$ which satisfies $|c|<\epsilon$. Then we put $g(t)=c_{0}t^6+\sum_{i=0}^{4}c_{6-i}t^i$ and note that $\cal{S}_{g}$ contains a $k$-rational point at infinity. Moreover, $S_{g}\simeq S_{h'}$, where $h'(t)=t^6g(1/t)$. From the Corollary \ref{cor:degree6} we get the result.
\end{proof}

The presented results motivate us to state the following:

\begin{conj}\label{conj1}
Let $k$ be a number field and $K/k$ be a cyclic extension of degree 3. Let $f\in k[t]$ be a polynomial of degree 6 and let us suppose that there exists a non-trivial $k$-rational point on $\cal{S}_{f}$. Then $\cal{S}_{f}$ is $k$-unirational.
\end{conj}

We finish this section with the following simple result.

\begin{thm}
Let $k$ be a field of characteristic 0 and let $K=k(\alpha)$, where $\alpha^3+b=0$ with $b\in k\setminus k^{3}$. Put $f(t)=t^{3m}+a_{2}t^{m}+a_{1}t+a_{0}\in k[t]$ with $a_{1}\neq 0$. Then the variety $\cal{S}_{f}$ is $k$-unirational.
\end{thm}
\begin{proof}
Let $F=F(X_{1},X_{2},X_{3},t)$ be a polynomial defining the variety $\cal{S}_{f}$. We put
\begin{equation*}
X_{1}=t^{m},\quad X_{2}=u,\quad X_{3}=\frac{a_{2}}{3bu}.
\end{equation*}
For $X_{i}$ defined in this way the polynomial $F$ (in the variable $t$) is of degree 1 with the root
\begin{equation*}
t=\phi(u)=-\frac{27 b^2 u^6+27 b a_0 u^3-a_2^3}{27 b a_1 u^3},
\end{equation*}
which under the assumption $a_{1}\neq 0$ is a non-constant element of $k(u)$. We thus see that the cubic surface $\cal{S}_{f\circ\phi}$ is $k(u)$-unirational and this implies the $k$-unirationality of $S_{f}$.
\end{proof}

\section{Solutions of the equation $\op{N}_{K/k}(X_{1},X_{2},X_{3})=f(t)$ for general cubic extension and $f$ of degree 6}\label{sec3}

We consider now the variety $\cal{S}_{f}$ given by the equation (\ref{degree3}) for a general extension $K/k$ of degree 3 and a monic polynomial $f\in k[t]$ of degree 6. We thus assume that $K=k(\alpha)$, where $\alpha$ is a root of an irreducible polynomial $h(x)=x^3+ax+b\in k[x]$ with $a\neq 0$. Unfortunately, in this case we were unable to prove the $k$-unirationality of $\cal{S}_{f}$ for all polynomials $f$ which satisfy $f(t)\neq f(\zeta_{3}t)$. However, we prove the following result.

\begin{thm}\label{thmnoncyclic}
Let $k$ be a field of characteristic 0 and put $K=k(\alpha)$, where $\alpha^3+a\alpha +b=0$ and $f(t)=t^6+a_{4}t^4+a_{1}t+a_{0}\in k[t]$ with $a_{1}a_{4}\neq 0$. Then the variety $\cal{S}_{f}$ given by {\rm (\ref{degree3})} is unirational over $k$.
\end{thm}
\begin{proof}
In this case the norm form takes the form $\op{N}_{K/k}=\op{N}_{K/k}(X_{1},X_{2},X_{3})$, where
\begin{equation*}
\op{N}_{K/k}=X_{1}^3- bX_{2}^3+ b^2X_{3}^3 + (aX_{2}+3bX_{3})X_{1}X_{2}-(2aX_{1}^2-a^2X_{1}X_{3}-abX_{2}X_{3})X_{3}.
\end{equation*}
Let $G=G(X_{1},X_{2},X_{3},t)$ be the polynomial defining the variety $\cal{S}_{f}$.
We use exactly the same approach as in the proof of Theorem \ref{degree4}. This time we just take $X_{1}=t^2+p$, where $p$ needs to be determined. We thus get $G(X_{1},X_{2},X_{3},t)=\sum_{i=0}^{4}C_{i}t^i$, where
\begin{equation*}
C_{2}=a^2X_3^2-4apX_3+aX_2^2+3bX_2X_3+3p^2,\quad C_{3}=0,\quad C_{4}=3p-a_{4}-2aX_{3}.
\end{equation*}
Eliminating $p$ from the equation $C_{4}=0$ we are left with the equation $C_{2}=0$ defining a curve, say $\cal{C}$, in the plane $(X_{2},X_{3})$. The equation for $\cal{C}$ can be rewritten in the form
\begin{equation*}
\cal{C}:\;(2a^2X_{3}-9bX_{2})^2=4a^2a_{4}^2+3(4a^3+27b^2)X_{2}^2.
\end{equation*}
The curve $\cal{C}$ is of genus 0 and has a rational point $(X_{2},X_{3})=(0,a_{4}/a)$ and thus can be parameterized by rational functions. A parametrization of $\cal{C}$ together with the expression for $p$ is given by
\begin{equation*}
X_{2}=\frac{4 a a_4 u}{3(4a^3+27b^2)-u^2},\quad X_{3}=\frac{a_4 \left(12 a^3+81 b^2+18 b u+u^2\right)}{a(3(4a^3+27b^2)-u^2)},\quad p=\frac{a_4+2aX_{3}}{3}.
\end{equation*}
For $X_{2}, X_{3}$ and $p$ chosen in this way we have the equality $DG(X_{1},X_{2},X_{3},t)=A_{0}+A_{1}t$, where $\op{deg}A_{0}=6$ and $D=A_{1}=-27a^3a_1(12 a^3+81b^2-u^2)^3$. From the assumption on $a_{1}$ we know that $DA_{1}\neq 0$. Careful analysis of the coefficients of the polynomial $A_{0}$ shows that if the coefficients of $f$ satisfy $a_{1}a_{4}\neq 0$ then the function $t=\phi(u)=-A_{0}/A_{1}$ satisfies $\phi\in k(u)\setminus k$. Thus, we have found a rational curve lying on $\cal{S}_{f}$. Finally, the same argument as at the end of the proof of Theorem \ref{degree4} gives $k$-unirationality of the variety $\cal{S}_{f}$.
\end{proof}

\begin{rem}
{\rm It is natural to ask whether the method employed in order to get the $k$-unirationality of the varieties considered in this paper can be used in other situations. More precisely, one can ask the following.

\begin{ques}\label{ques2}
 Let $f\in k[t]$. How general an indecomposable form $F\in k[X_{1},X_{2},X_{3}]$ of degree 3 can be such that a variety defined by the equation $F(X_{1},X_{2},X_{3})=f(t)$ is unirational over $k$ for most choices of $f$ of fixed degree?
\end{ques}

For example, let us consider the case of a monic polynomial $f\in k[t]$ of degree 6. It would be rather unexpected if taking the form
\begin{equation*}
F(X_{1},X_{2},X_{3})=X_{1}^3+aX_{2}^3+bX_{3}^3+(cX_{1}+dX_{2}+eX_{3})X_{2}X_{3},
\end{equation*}
we could prove the $k$-unirationality of the hypersurface defined by the equation
$$\cal{S}:\;F(X_{1},X_{2},X_{3})=f(t),$$
where $f(t)=t^6+\sum_{i=0}^{4}a_{i}t^{i}\in k[t]$ and $a,b,c,d,e\in k$ satisfy certain conditions. We note that for a generic choice of $a, b, c, d, e\in k$ the form $F$ is absolutely irreducible, i.e. is irreducible as a polynomial in $\bar{k}[X_{1},X_{2},X_{3}]$. Let $G(X_{1},X_{2},X_{3},t)=F(X_{1},X_{2},X_{3})-f(t)$ be the polynomial defining the hypersurface $\cal{S}$. In order to verify the $k$-unirationality of $\cal{S}$, it is enough to take
\begin{equation}\label{xyz:genform}
X_{1}=t^2+\frac{a_{4}}{3},\quad X_{2}=\frac{b_3-b u^3}{c u},\quad X_{3}=ut+\frac{u \left(3 b e u^4-3 a_3 e u-a_4^2 c+3 a_2 c\right)}{3 c \left(2 b u^3+a_3\right)}.
\end{equation}
Indeed, for $X_{1},X_{2},X_{3}$ chosen in this way we note an equality $DG(X_{1},X_{2},X_{3},t)=C_{1}t+C_{0}$, where $C_{0}, C_{1}\in k[u]$ depend on the coefficients $a, b, c, d, e$ and $a_{i}$ for $i=0,\ldots,4$. Moreover, we have $D=27c^3u^3(2bu^3+a_{3})^{3}$. If $C_{0}C_{1}\neq 0$ as a polynomial in $k[u]$, we get a solution $t=\phi(u)=-C_{0}/C_{1}$. We have $\op{deg}C_{1}=17$ and $\op{deg}C_{0}=18$. The expression for $t$ together with the expressions for $X_{1}, X_{2}, X_{3}$ given by (\ref{xyz:genform}) yield a parametrization (in the parameter $u$) of a rational curve on $\cal{S}$ with $f(\phi(u))\neq 0$. The existence of a rational curve lying on $\cal{S}$ allows us to define a rational base change $t=\phi(u)$. Then the (cubic) surface $\cal{S}_{\phi}:\;F(X_{1},X_{2},X_{3})=f(\phi(u))$ (treated as a surface over the field $k(u)$) contains a smooth $k(u)$-rational point $P$ with coordinates given by (\ref{xyz:genform}) and thus $\cal{S}_{\phi}$ is $k$-unirational over $k(u)$. As an immediate consequence we get the $k$-unirationality of $\cal{S}$ over $k$.

It is possible to write explicit conditions on the coefficients of the polynomial $f$ and the form $F$ which will guarantee that $\phi\in k(u)\setminus k$. For example, if  $abcea_{3}\neq 0$ then $\phi\in k(u)\setminus k$.
}
\end{rem}

\bigskip


\vskip 1cm

\noindent Maciej Ulas, Jagiellonian University, Faculty of Mathematics and Computer Science, Institute of
Mathematics, {\L}ojasiewicza 6, 30-348 Krak\'ow, Poland; email:
maciej.ulas@uj.edu.pl


\begin{thebibliography}{100}


\bibitem{CSal} J. L. Colliot-Th\'{e}l\`{e}ne, P. Salberger, {\it Arithmetic on some singular cubic hypersurfaces}, Proc. London Math. Soc. (3), 58 (1989), 519--549.

\bibitem{CTSSD1} J. L. Colliot-Th\'{e}l\`{e}ne, J.-J. Sansuc, P. Swinnerton-Dyer, {\it Intersections
of two quadrics and Ch\^{a}telet surfaces. I}, J. reine angew. Math., 373 (1987) 37–-107.

\bibitem{CTSSD2} J. L. Colliot-Th\'{e}l\`{e}ne, J.-J. Sansuc, P. Swinnerton-Dyer. {\it Intersections
of two quadrics and Ch\^{a}atelet surfaces. II},  J. reine angew. Math., 374 (1987), 72–-168.

\bibitem{CSS} J. L. Colliot-Th\'{e}l\`{e}ne, A. N. Skorobogatov, P. Swinnerton-Dyer, {\it Rational points and zero-cycles on fibred varieties: Schinzel's hypothesis and Salberger's device}, J. Reine Angew. Math., 495 (1998), 1--28.

\bibitem{HS} R. D. Heath-Brown, A. N. Skorobogatov, {\it Rational solutions of certain equations involving norms}, Acta Math., 189 (2002), 161-177.

\bibitem{Mes1} J. F. Mestre, {\it Annulation, par changement de variable, d'\'{e}l\'{e}ments de $\op{Br}_{2}(k(x))$ ayant huit p\^{o}les, \`{a} r\'{e}sidu constant}, C. R. Acad. Sci. Paris S\'{e}r. I Math. 319 (1994), no. 11, 1147--1149.

\bibitem{Mes2} J. F. Mestre, {\it Annulation, par changement de variable, d'\'{e}l\'{e}ments de $\op{Br}_{2}(k(x))$ ayant quatre p\^{o}les}, C. R. Acad. Sci. Paris S\'{e}r. I Math. 319 (1994), no. 6, 529--532.

\bibitem{Mes3} J. F. Mestre, {\it Annulation, par changement de variable, d'\'{e}l\'{e}ments de $\op{Br}_{2}(k(x))$ ayant cinq p\^{o}les}, C. R. Acad. Sci. Paris S\'{e}r. I Math. 322 (1996), no. 6, 503--505.

\bibitem{VarBir} A. V\'{a}rilly-Alvarado, B. Viray, {\it Higher dimensional analogoues of Ch\^{a}telet surfaces}, Bull. London Math. Soc., 44 (2012) 125--135.

\end{thebibliography}
\end{document}